\documentclass[12pt,letterpaper]{article}

\usepackage{amssymb,amsfonts,amscd,amsthm}
\usepackage[all,arc]{xy}
\usepackage{enumerate, bbm}
\usepackage{mathrsfs}
\usepackage{tikz}
\usepackage[left=.9in,top=.9in,right=.9in,bottom=.9in]{geometry}
\usepackage{mathtools}
\usepackage{hyperref}
\usepackage{color}
\usepackage{graphicx}
\usepackage{enumitem} 
\graphicspath{ {TexImages/} }

\newtheorem{thm}{Theorem}[section]
\newtheorem{cor}[thm]{Corollary}
\newtheorem{prop}[thm]{Proposition}
\newtheorem{lem}[thm]{Lemma}
\newtheorem{claim}[thm]{Claim}

\newtheorem{prob}[thm]{Problem}

\hypersetup{
	colorlinks,
	citecolor=blue,
	filecolor=blue,
	linkcolor=blue,
	urlcolor=blue,
	linktocpage
}

\setenumerate[1]{label=\thesection.\arabic*.} 
\setenumerate[2]{label*=\arabic*.} 

\newcommand{\E}{\mathbb{E}}

\newcommand{\al}{\alpha}
\newcommand{\be}{\beta}

\newcommand{\ep}{\varepsilon}
\newcommand{\lam}{\lambda}

\newcommand{\Om}{\Omega}

\newcommand{\Del}{\Delta}
\newcommand{\pa}{\partial}

\renewcommand{\l}{\left}
\renewcommand{\r}{\right}

\newcommand{\half}{\frac{1}{2}}

\renewcommand{\1}{\mathbbm{1}}
\newcommand{\sm}{\setminus}
\newcommand{\sub}{\subseteq}

\renewcommand{\c}[1]{\mathcal{#1}}

\newcommand{\ol}[1]{\overline{#1}}
\newcommand{\tr}[1]{\textrm{#1}}
\newcommand{\rec}[1]{\frac{1}{#1}}
\newcommand{\f}[2]{\frac{#1}{#2}}
\newcommand{\floor}[1]{\l\lfloor #1\r\rfloor}

\newcommand{\Graph}[2]{\includegraphics[width=#1\textwidth]{#2}}
\newcommand{\mr}[1]{\mathrm{#1}}

\newcommand{\ex}{\mr{ex}}

\newcommand{\HI}[1]{\c{H}(#1)}

\newcommand{\B}[2]{\c{C}_{#1}^{#2}}
\newcommand{\BB}[2]{\widehat{\c{C}}_{#1}^{#2}}
\renewcommand{\S}[2]{\widetilde{\c{S}}_{#1}^{#2}}
\newcommand{\F}{\mathbf{F}}
\tikzstyle{vertex} = [fill,shape=circle,node distance=50pt]
\tikzstyle{edge} = [fill,opacity=.5,fill opacity=.5,line cap=round, line join=round, line width=50pt]
\tikzstyle{elabel} =  [fill,shape=circle,node distance=10pt]

\pgfdeclarelayer{background}
\pgfsetlayers{background,main}
\begin{document}

\title{Relative Tur\'an Numbers for Hypergraph Cycles}

\author{Sam Spiro\thanks{Department of Mathematics, University of California, San Diego, 9500 Gilman Drive, La Jolla, CA 92093-0112, USA. E-mail: sspiro@ucsd.edu. This material is based upon work supported by the National Science Foundation Graduate Research Fellowship under Grant No. DGE-1650112.}\and
	Jacques Verstra\"ete\thanks{Department of Mathematics, University of California, San Diego, 9500 Gilman Drive, La Jolla, CA 92093-0112, USA. E-mail: jacques@ucsd.edu. Research supported by the National Science Foundation Awards DMS-1800332 and DMS-1952786,
		and by the Institute for Mathematical Research (FIM) of ETH Z\"urich.}}

\maketitle

\begin{abstract}
	For an $r$-uniform hypergraph $H$ and a family of $r$-uniform hypergraphs $\mathcal{F}$, the relative Tur\'{a}n number $\mathrm{ex}(H,\mathcal{F})$ is the maximum number of edges in an $\mathcal{F}$-free subgraph of $H$. In this paper we give lower bounds on $\mathrm{ex}(H,\mathcal{F})$ for certain families of hypergraph cycles $\mathcal{F}$ such as Berge cycles and loose cycles.  In particular, if $\mathcal{C}_\ell^3$ denotes the set of all $3$-uniform Berge $\ell$-cycles and $H$ is a 3-uniform hypergraph with maximum degree $\Delta$, we prove
	\[\mathrm{ex}(H,\mathcal{C}_4^{3})\ge \Delta^{-3/4-o(1)}e(H),\]
	\[\mathrm{ex}(H,\mathcal{C}_5^{3})\ge \Delta^{-3/4-o(1)}e(H),\]
	and these bounds are tight up to the $o(1)$ term.
\end{abstract}

\section{Introduction}
Let $\c{F}$ be a family of $r$-uniform hypergraphs, or $r$-graphs for short.  The {\em Tur\'{a}n number} $\ex(n,\c{F})$ is defined to be the maximum number of edges in an $\c{F}$-free $n$-vertex $r$-graph. The Tur\'{a}n numbers are a central object of study in extremal graph theory, dating back to Mantel's Theorem~\cite{M} and Tur\'{a}n's Theorem~\cite{T}. A more general problem involves studying the {\em relative Tur\'an number} $\ex(H,\c{F})$, which is the maximum number of edges in an $\c{F}$-free subgraph of an $r$-graph $H$, and we will say that $H$ is a \textit{host} hypergraph.  For example, when $H$ is $K_n^r$ (the complete $r$-graph on $n$ vertices), we simply recover the original Tur\'{a}n number.  

The study of relative Tur\'an numbers was advanced  by Foucaud, Krivelevich, and Perarnau~\cite{FKP} and independently by Briggs and Cox~\cite{BriggsCox}.  Many results have been obtained for relative Tur\'an numbers, both when $H$ is a general host, as well as for random hosts \cite{GSTZ,NSV, PR, SV-Girth, SV-K22}.

In this paper we consider relative Tur\'an numbers for hypergraph cycles, and in particular for Berge cycles.  If $F$ and $F'$ are hypergraphs with $V(F)\sub V(F')$, we say that $F'$ is a \textit{Berge}-$F$ if there exists a bijection $\phi:E(F)\to E(F')$ such that $e\sub \phi(e)$ for all $e\in E(F)$.  We let $\B{\ell}{r}$ denote the set of all $r$-uniform Berge-$F$ where $F$ is an $\ell$-cycle.  When $\ell=2$, $F$ is a double edge, and any $r$-graph which is $\B{2}{r}$-free is said to be \textit{linear}. We let $\B{[\ell]}{r}:=\bigcup_{\ell'=2}^\ell\B{\ell'}{r}$, and any $r$-graph which is $\B{[\ell]}{r}$-free is said to have \textit{girth} $\ell+1$.  In \cite{SV-Girth}, we conjectured
\begin{equation}\label{eq:girth}
\ex(n,\B{[\ell]}{r})\ge n^{1+\rec{\floor{\ell/2}}-o(1)},\tag{$*$}
\end{equation}
which is a natural generalization of a conjecture of Erd\H{o}s and Simonovits~\cite{ES} for graphs.  For $\ell,r\ge 3$, Gy\H{o}ri and Lemons \cite{GL} proved $\ex(n,\B{\ell}{r})=O(n^{1+\f{1}{\floor{\ell/2}}})$, so up to the $o(1)$ term, \eqref{eq:girth} would be best possible. This conjecture is known to hold for $\ell=3$ and all $r$ due to work of Ruzsa and Szemer\'edi~\cite{RS} and Erd\H{o}s, Frankl, and R\"odl~\cite{EFR}, and for $\ell=4$ and all $r$ due to Lazebnik and the second author~\cite{LV} and Timmons and the second author~\cite{TV}.  We emphasize that for $r\ge 3$, it is known that the $o(1)$ term in \eqref{eq:girth} is necessary in general, see Ruzsa and Szemer\'edi~\cite{RS} and Conlon, Fox, Sudakov, and Zhao \cite{CFSZ}.

Our main result is the following, where we recall that an $r$-graph $S$ is a \textit{sunflower} if there exists a set $K$ called the \textit{kernel} such that $e\cap e'=K$ for every pair of distinct edges $e,e'\in E(S)$.  Throughout this paper, the maximum degree $\Del$ of a hypergraph $H$ is the maximum degree of any vertex of $H$, and all of our asymptotic statements are with respect to  $\Delta$ tending towards infinity.
\begin{thm}\label{thm:BGen}
	 Let $\BB{[\ell]}{r}$ consist of all the elements of $\B{[\ell]}{r}$ which are not sunflowers. If $\ell,r\ge 3$ are such that \eqref{eq:girth} holds, then for all $r$-graphs $H$ with maximum degree at most $\Del$, we have 
	\begin{equation}\label{eq:BB}\ex(H,\BB{[\ell]}{r})\ge \Del^{-1+\f{1}{(r-1)\floor{\ell/2}}-o(1)} \cdot e(H),\end{equation}
	and this bound is tight up to the $o(1)$ term in the exponent for $H=K_{\Del^{1/(r-1)}}^r$.
\end{thm}

The $r=2$ case of Theorem~\ref{thm:BGen} was proven by Perarnau and Reed~\cite{PR}.  We note that that the bound of Theorem~\ref{thm:BGen} assuming \eqref{eq:girth} can be shown to be tight up to the $o(1)$ term in the exponent by considering $H=K_{\Del^{1/(r-1)}}^r$.

In Theorem~\ref{thm:BGen}, it is necessary to consider Berge cycles without sunflowers.  Indeed, if $S\in \B{\ell}{r}$ is a sunflower with kernel of size $k$, and if $H$ is a sunflower with $\Del$ edges and kernel of size $k$, then $\ex(H,S)=\ell-1=O(\Del^{-1})\cdot e(H)$.

Using Theorem~\ref{thm:BGen}, known results for classical Tur\'an numbers, the fact that no member of $\B{3}{3},\B{4}{3},\B{4}{4}$ is a sunflower, and the observation that $\ex(H,\c{F}')\ge \ex(H,\c{F})$ whenever $\c{F}'\sub \c{F}$, we deduce the following.
\begin{cor}\label{cor:B}
	If $H$ is a 3-graph with maximum degree at most $\Del$, then \begin{align*}
		\ex(H,\B{3}{3})&\ge \Del^{-1/2-o(1)}\cdot e(H),\\
		\ex(H,\B{4}{3})&\ge \Del^{-3/4-o(1)}\cdot e(H).\end{align*}
	If $H$ is a 4-graph with maximum degree at most $\Del$, then 
	\[\ex(H,\B{4}{4})\ge \Del^{-5/6-o(1)}\cdot e(H).\]
\end{cor}
Again all of these bounds are tight up to the $o(1)$ term in the exponent.  Using a similar argument, we will show in Section~\ref{sec:B} that for any 3-graph $H$ with maximum degree at most $\Del$,
\[\ex(H,\B{5}{3})\ge \Del^{-3/4-o(1)}\cdot e(H),\]
which does not follow from Theorem~\ref{thm:BGen} since the order of magnitude of $\ex(n,\B{[5]}{3})$ is not known, see \cite{CFSZ}.

We next consider the {\em loose cycle}  $C_\ell^r$, which is the $r$-graph which has $\ell$ edges $e_1,\ldots,e_\ell$ such that $e_i\cap e_{i+1}=\{v_i\}$ for $1\le i\le \ell$ with all $v_i$ distinct and the indices written cyclically, and such that $e_i\cap e_j=\emptyset$ for any other pair of indices $i\ne j$.   

While the classical Tur\'an numbers for loose cycles are known exactly (see \cite{FF,FJ,KMV}), it appears to be difficult to find tight bounds for relative Tur\'an numbers for loose cycles.  In large part this seems to be because, unlike in Theorem~\ref{thm:BGen}, the clique $K_{\Del^{1/(r-1)}}^r$ does not give tight bounds for $\ex(H,C_\ell^r)$ in general.  Indeed, by using recent results of Mubayi and Yepreyman~\cite{MY}, we show that certain random hypergraphs give stronger bounds for $\ex(H,C_\ell^r)$ compared to the clique when $\ell$ is even:

\begin{thm}\label{thm:loose}~
	Let $\ell\ge 3$.  If $H$ is a 3-graph with maximum degree at most $\Del$, then \[\ex(H,C_\ell^3)\ge \Del^{-1+\rec{\ell}-o(1)}\cdot e(H).\]
		
	Moreover, if $\ell$ is even, then there exists a $3$-graph $H$ with maximum degree at most $\Del$ and \[\ex(H,C_{\ell}^{3})\le \Del^{-1+\rec{\ell-1}+o(1)}\cdot e(H).\]
	
\end{thm}
It is possible to extend our arguments to $r$-graphs, but the gap between the lower and upper bound grows considerably with $r$.  When the host $H$ is linear, we improve the bounds of Theorem~\ref{thm:loose} to give tight results for all $r$ when $\ell$ is even.
\begin{prop}\label{prop:linear}
	Let $\ell\ge 4$ be even and $r\ge 3$.   If $H$ is a linear $r$-graph with maximum degree at most $\Del$, then \[\ex(H,C_{\ell}^{r})\ge \Del^{-1+\rec{\ell-1}-o(1)}\cdot e(H).\]
	Moreover, there exists a linear $r$-graph $H$ with maximum degree at most $\Del$ and \[\ex(H,C_{\ell}^{r})\le \Del^{-1+\rec{\ell-1}+o(1)}\cdot e(H).\]
\end{prop}

The last hypergraph cycle we consider is $\F$, which is the 3-uniform Berge 4-cycle depicted in Figure~\ref{fig}.  The Tur\'an number for the hypergraph $\F$ is well studied \cite{ErdosDisjoint,FurediDisjoint, MubayiDisjoint, MVDisjoint, PV}.  By Corollary~\ref{cor:B}, we have $\ex(H,\F)\ge \Del^{-3/4-o(1)}e(H)$ for all hosts $H$ with maximum degree at most $\Del$.  We improve this lower bound as follows.

\begin{figure}
	\begin{center}
		\Graph{.34}{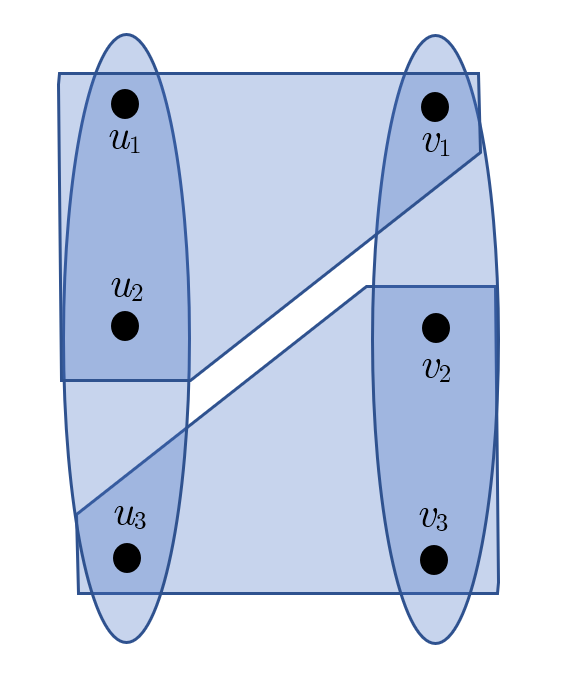}
	\end{center}
	\caption{The 3-graph $\F$ with edges $\{u_1,u_2,u_3\},\{u_1,u_2,v_1\},\{u_1,v_2,u_3\},\{v_1,v_2,v_3\}$}\label{fig}
\end{figure}

\begin{thm}\label{thm:disjoint}
	If $H$ is a 3-graph with maximum degree at most $\Del$, then \begin{align*} \ex(H,\F)\ge \Del^{-3/5-o(1)}\cdot e(H). \end{align*} 
\end{thm}
We do not have a matching upper bound for $\ex(H,\F)$, see the concluding remarks for further discussions.

\subsection{Organization and Notation} 
All of our proofs for lower bounding $\ex(H,\c{F})$ follow the same basic strategy which we briefly outline here.  By considering a dense subgraph of the host hypergraph, we can always assume that $H$ is $r$-partite.  If $H$ has small codegrees, i.e. if no set of vertices of $H$ is contained in many edges, then we apply the method of random homomorphisms developed in Section~\ref{sec:hom}.  The idea here is to take a dense $\c{F}$-free hypergraph $J$ and define a random mapping $\chi:V(H)\to V(J)$ and use $\chi$ to determine a large $\c{F}$-free subgraph of $H$. If $H$ has large codegrees then the above approach is not effective, and in this case we need tools which we develop in Section~\ref{sec:co}.  The idea here is that there will be some $k$ such that many sets of size $k$ in $H$ will be contained in many edges.  Given this, we try and find a subgraph of $H$ which induces a matching on the first $k$ parts.

After we develop these general techniques, we move on to prove our theorems.  We prove our main result for Berge cycles, Theorem~\ref{thm:BGen}, in Section~\ref{sec:B}, we prove Theorem~\ref{thm:disjoint} in Section~\ref{sec:disjoint}, and our results for loose cycles, Theorem~\ref{thm:loose} and Proposition~\ref{prop:linear}, are proven in Section~\ref{sec:L}. Concluding remarks and further problems are given in Section~\ref{sec:concluding}.

We gather some notation and definitions that we use throughout the text. A set of size $k$ will be called a {\em $k$-set}.  If $H$ is an $r$-graph, the number of edges containing a $k$-set $\{u_1,\ldots,u_k\}$ is called the {\em $k$-degree} of the set and is denoted by $d_H(u_1,\ldots,u_k)$.  Throughout the text we omit ceilings and floors whenever these are not crucial.  We often make use of the following basic fact due to Erd\H{o}s and Kleitman~\cite{EK}: every $r$-graph $H$ has an $r$-partite subgraph with at least $r^{-r}e(H)$ edges.  

We say that an $r$-graph $F$ is {\em non-linear} if two of its edges intersect in more than one vertex, or equivalently if it contains an element of $\B{2}{r}$ as a subgraph.   Given an $r$-graph $F$, we say that $A=\{v_1,\ldots,v_\ell\}\sub V(F)$ is a \textit{core set with respect to} $e_1,\ldots,e_\ell\in E(F)$ if $v_i,v_{i+1}\in e_i$ for all $1\le i<\ell$ and if $v_1,v_\ell\in e_\ell$, and we simply say that $A$ is a \textit{core set} if such a choice of edges exist.  Observe that the existence of a core set $\{v_1,\ldots,v_\ell\}$ in $F$ is equivalent to saying that $F$ contains an element of $\B{\ell}{r}$ as a subgraph.

\section{Hosts with Small Codegrees}\label{sec:hom}
If $\chi$ is a map from vertices of $H$ and $e=\{u_1,\ldots,u_r\}\in E(H)$, we define the set  $\chi(e)=\{\chi(u_1),\ldots,\chi(u_r)\}$.  If $F$ and $F'$ are $r$-uniform  hypergraphs, then we say that a map $\chi:V(F)\to V(F')$ is a \textit{local isomorphism} if $\chi$ is a homomorphism\footnote{That is, if for every $e\in E(F)$ we have $\chi(e)\in E(F')$.} and if $\chi(e)\ne \chi (f)$ whenever $e,f\in E(F)$ with $e\cap f\ne \emptyset$. We define $\HI{F}$ to be the set of $F'$ for which there exists a local isomorphism $\chi:V(F)\to V(F')$.  If $\c{F}$ is a family of $r$-graphs we define $\HI{\c{F}}=\bigcup_{F\in \c{F}}\c{H}(F)$.  Note that in general $\HI{F}$ will be infinite. Our main result for this section is the following general result.

\begin{prop}\label{prop:homGen}
	Let $H$ be an $r$-graph, and for all $k<r$ let $\Del_k$ denote the maximum $k$-degree of $H$.  If $t\ge r^24^r \Del_k^{1/(r-k)}$ for all $k$, then for any family of $r$-graphs $\c{F}$, \[\ex(H,\c{F})\ge \ex(t,\HI{\c{F}})t^{-r}\cdot e(H).\]
\end{prop}
\begin{proof}
	Let $J$ be an extremal $\HI{\c{F}}$-free $r$-graph on $t$ vertices and $\chi:V(H)\to V(J)$ chosen uniformly at random.  Let $H'\sub H$ be the random subgraph which keeps the edge $e\in E(H)$ if 
	\begin{itemize}
		\item[$(1)$] $\chi(e)\in E(J)$, and
		\item[$(2)$] $\chi(e)\ne \chi(f)$ for any other $f\in E(H)$ with $|e\cap f|\ne 0$.  
	\end{itemize}
	
	We claim that $H'$ is $\c{F}$-free.  Indeed, assume $H'$ contained a subgraph $F$ isomorphic to some element of $\c{F}$.  Let $F'$ be the subgraph of $J$ with $V(F')=\{\chi(u):u\in V(F)\}$ and $E(F')=\{\chi(e):e\in E(F)\}$, and we note that $F\sub H'$ implies that each edge of $F$ satisfies (1), so every element of $E(F')$ is an edge in $J$.  By conditions (1) and (2), $\chi$ is a local isomorphism from $F$ to $F'$, so $F'\in \HI{\c{F}}$, a contradiction to $J$ being $\HI{\c{F}}$-free.
	
	It remains to compute $\E[e(H')]$.  Fix $e\in E(H)$. Let $A$ denote the event that (1) is satisfied, let $E_k=\{f\in E(H):|e\cap f|=k\}$, and let $B$ denote the event that $\chi(f)\not\sub \chi(e)$ for all $f\in \bigcup_{k=1}^{r-1} E_k$, which in particular implies (2) for the edge $e$.  It is not too difficult to see that $\Pr[A]= r!e(J)t^{-r}$, and that for all $f\in E_k$ we have $\Pr[\chi(f)\sub \chi(e)|A]=(r/t)^{r-k}$.  We note that $|E_k|\le {r\choose k} \Del_k\le 2^r \Del_k$ for all $k$, since $e$ has at most ${r\choose k}$ subsets of size $k$ and each has $k$-degree at most $\Del_k$.   Thus taking a union bound we find \[\Pr[B|A]\ge 1-\sum_{k} |E_k|(r/t)^{r-k}\ge 1- \sum_{k} 2^r \Del_k(r/t)^{r-k}\ge 1-\sum_{k} r^{-1}2^{-r}\ge \half,\]
	where the second to last inequality used $(r 4^r)^{k-r}\ge r^{-1}4^{-r}$ for $k<r$.  Thus
	\[\Pr[e\in E(H')]=\Pr[A]\cdot \Pr[B|A]\ge r! e(J)t^{-r}\cdot \half \ge e(J)t^{-r},\]
	and linearity of expectation gives $\E[e(H')]\ge e(J)t^{-r}\cdot e(H)=\ex(t,\HI{\c{F}})t^{-r}\cdot e(H)$.  Thus there exists some subgraph $H'\sub H$ with this many edges which is $\c{F}$-free, giving the result.
\end{proof}

We use the following lemma in order to apply Proposition~\ref{prop:homGen} to familes of Berge cycles.

\begin{lem}\label{lem:homB}
~
\begin{itemize}
	\item[$(1)$] If $F$ is a non-linear $r$-graph, then $\HI{F}$ consists of non-linear $r$-graphs.
	\item[$(2)$] Every element of $\HI{\B{5}{3}}$ contains an element of $\B{2}{3}\cup \B{5}{3}$ as a subgraph.
	\item[$(3)$] For $\ell,r\ge 2$, every element of $\HI{\B{[\ell]}{r}}$ contains an element of $\B{[\ell]}{r}$ as a subgraph.
\end{itemize}
\end{lem}
\begin{proof}
	For (1), let $e_1,e_2\in E(F)$ be edges and $u\ne v$ such that $u,v\in e_1\cap e_2$.  Let $\chi:V(F)\to V(F')$ be a local isomorphism, and because $e_1\cap e_2\ne  \emptyset$ we have $\chi(e_1)\ne \chi(e_2)$.  We have $\chi(u)\ne \chi(v)$, as otherwise $|\chi(e_1)|<r$, contradicting $\chi$ being a homomorphism.  Thus the distinct edges $\chi(e_1),\chi(e_2)$ intersect in at least two vertices in $F'$ as desired.
	
	For (2), let $F\in \B{5}{3}$ and let $A=\{v_1,\ldots,v_5\}$ be a core set with respect to the edges $e_1,\ldots,e_5\in E(F)$. Consider a local isomorphism $\chi:V(F)\to V(F')$.  Note that $\chi(e_i)\ne \chi(e_{i+1})$ for any $i$ (where here and throughout this proof we write our indices mod 5) since $e_i\cap e_{i+1}\ne \emptyset$ and $\chi$ is a local isomorphism.  Similarly, observe that $\chi(v_i)\ne \chi(v_{i+1})$ for any $i$, otherwise this would imply $|\chi(e_i)|<3$, contradicting $\chi(e_i)$ being an edge of $F'$.
	
	If $\chi(e_{i})=\chi(e_{i+2})$, then the distinct edges $\chi(e_i),\chi(e_{i+1})$ both contain the distinct vertices $\chi(v_{i+1}),\chi(v_{i+2})$, so $F'$ contains an element of $\B{2}{3}$ as a subgraph.  Thus we will assume $\chi(e_i)\ne \chi(e_{i+2})$ for all $i$, and this together with $\chi(e_i)\ne \chi(e_{i+1})$ implies that all the edges $\chi(e_i)$ are distinct.  Similarly if $\chi(v_i)=\chi(v_{i+2})$, then $\chi(e_i),\chi(e_{i+1})$ show that $F'$ contains an element of $\B{2}{3}$.  We conclude that all of the $\chi(v_i)$ vertices are distinct, and thus these form a core set for the distinct edges $\chi(e_1),\ldots,\chi(e_5)$, so $F'$ contains an element of $\B{5}{3}$ as a subgraph.
	
	Case (3) is proven in  \cite[Lemma 4.2]{SV-Girth} using an argument similar to the proof of (2).
\end{proof}

\section{Hosts with Large Codegrees}\label{sec:co}
As a motivating example for the method of this section, consider the case that $H$ is a complete tripartite 3-graph on $V_1\cup V_2\cup V_3$ with $|V_1|=|V_2|=o(|V_3|)$.  If one wanted to find some $H'\sub H$ avoiding the loose 4-cycle $C_4^{3}$, then an approach would be to take a perfect matching $M$ on $V_1\cup V_2$ and then to take $H'$ to be all the 3-sets containing a pair from $M$.  It turns out that $H'$ is $C_4^{3}$-free and has many edges.  In fact, it will avoid every element of $\B{4}{3}$ except for the 3-graph $\F$ defined in Figure~\ref{fig}. To deal with $\F$, instead of taking $H'$ to have all edges containing an edge of $M$, we take a subset of these edges such that the graph induced by $V_2\cup V_3$ in $H'$ is $C_4$-free, and this will turn out to be enough to give the desired result.

\subsection{Induced $k$-graphs}
We recall that a hypergraph is a \textit{matching} if it is a disjoint union of edges. Given an $r$-partite $r$-graph $F$ on $V_1\cup\cdots\cup V_r$ and a $k$-set $I$, we define $\pa F\l[\bigcup_{i\in I}V_i\r]$ to be the $k$-graph $F'$ with 
\[V(F')=\bigcup_{i\in I}V_i,\hspace{2em} E(F')=\{e\cap \bigcup_{i\in I}V_i:e\in E(F)\},\]
and we call this the {\em$k$-graph induced by} $\bigcup_{i\in I}V_i$.  Let $\c{R}(F)$ denote the set of all $r$-partitions of $F$, and define \[\textstyle{\c{P}_k(F)=\{\pa F\l[\bigcup_{i\ge k}V_i\r]:(V_1,\ldots, V_r)\in \c{R}(F),\ \pa F\l[\bigcup_{i\le k}V_i\r]\tr{ is a matching}\}}.\] 
That is, $\c{P}_k(F)$ consists of all $(r-k+1)$-graphs which can be obtained by first taking an $r$-partition of $F$ such that its first $k$ parts induce a matching, and then taking the $(r-k+1)$-graph induced by $\bigcup_{i\ge k}V_i$.  For example, $C_4\in \c{P}_2(\F)$ since one can consider the tripartition $V_1=\{u_1,v_3\},\ V_2=\{u_2,v_2\},\ V_3=\{u_3,v_1\}$, and it is not too difficult to show that $ \c{P}_2(\F)=\{C_4\}$.   If $\c{F}$ is a family of $r$-graphs we define $\c{P}_k(\c{F})=\bigcup_{F\in \c{F}}\c{P}_k(F)$.   Note that if $F$ is not $r$-partite then $\c{P}_k(F)=\emptyset$.

Our main result for this section is the following, which roughly says that one can prove effective lower bounds on $\ex(H,\c{F})$ if $H$ has many edges containing sets with large $k$-degrees and if one knows effective lower bounds for $\ex(G,\c{P}_k(\c{F}))$ with $G$ any $(r-k+1)$-graph.

\begin{prop}\label{prop:coMatch}
	Let $k<r$, $0\le \be\le 1$, and let $\c{F}$ be a family of $r$-graphs such that $\ex(G,\c{P}_k(\c{F}))=\Om(\widetilde{\Delta}^{-\be}) e(G)$ for any $(r-k+1)$-graph $G$ with maximum degree at most $\widetilde{\Delta}$.  If $H$ is an $r$-partite $r$-graph with maximum degree at most $\Del$ and at least half of the edges of $H$ contain a $k$-set with $k$-degree at least $D$, then \[\ex(H,\c{F})=\Om\l(\f{D^{1-\be}}{\Del(\log \Del)^{7}}\r)\cdot e(H).\]
\end{prop}

We note that we make no attempt at optimizing the log factor in Proposition~\ref{prop:coMatch}. To prove this result, we establish a few lemmas.

\begin{lem}\label{lem:LFree}
	Let $H$ be an $r$-partite $r$-graph on $V_1\cup \cdots \cup V_r$ and $\c{F}$ a family of $r$-graphs.  If $ \pa H\l[\bigcup_{i\le k}V_i\r]$ is a matching and $ \pa H\l[\bigcup_{i\ge k}V_i\r]$ is $\c{P}_k(\c{F})$-free, then $H$ is $\c{F}$-free.
\end{lem}
\begin{proof}
	Assume $H$ contains a subgraph $F$ isomorphic to some element of $\c{F}$, and consider the $r$-partition $U_1\cup \cdots \cup  U_r$ of $F$ defined by $U_i=V(F)\cap V_i$.  Then $\pa F\l[\bigcup_{i\in I}U_i\r]\sub \pa H\l[\bigcup_{i\in I}V_i\r]$ for all sets $I$.  In particular, $\pa F\l[\bigcup_{i\le k}U_i\r]\sub \pa H\l[\bigcup_{i\le k}V_i\r]$ is a matching, and hence $\pa F\l[\bigcup_{i\ge k}U_i\r]\in\c{P}_k(\c{F})$ by definition.  This contradicts $\pa F\l[\bigcup_{i\ge k}U_i\r]\sub \pa H\l[\bigcup_{i\ge k}V_i\r]$ and $\pa H\l[\bigcup_{i\ge k}V_i\r]$ being  $\c{P}_k(\c{F})$-free.
\end{proof}

\begin{lem}\label{lem:match}
	If $H$ is an $r$-graph with maximum degree at most $\Del>0$, then $H$ contains a matching with at least $e(H)/r\Del$ edges.
\end{lem}
The proof of Lemma~\ref{lem:match} is a straightforward greedy argument and we omit its proof.  Lastly, we use the Chernoff bound~\cite{ProbMeth}.
\begin{prop}[\cite{ProbMeth}]\label{prop:Chern}
	Let $X$ denote a binomial random variable with $N$ trials and probability $p$ of success.  For any $\lam>0$ we have \begin{equation}\label{eq:Chern}\Pr[X-pN>\lam pN]\le e^{-\lam^2 pN/2}.\end{equation}
\end{prop}

With this we can prove a version of Proposition~\ref{prop:coMatch} under the additional hypothesis that every $k$-set of $\bigcup_{i\le k} V_i$ is contained in either 0 or roughly $D$ edges.

\begin{prop}\label{prop:coMatchTech}
	 Let $k<r$, $0\le \be\le 1$, and let $\c{F}$ be a family of $r$-graphs such that $\ex(G,\c{P}_k(\c{F}))=\Om(\widetilde{\Delta}^{-\be}) e(G)$ for any $(r-k+1)$-graph $G$ with maximum degree at most $\widetilde{\Delta}$.  Let $H$ be an $r$-partite $r$-graph on $V_1\cup \cdots\cup V_r$ with maximum degree at most $\Del$ and let $D$ be such that $D\le \Del/\log \Del$ and such that if $\{u_1,\ldots,u_r\}\in E(H)$ with $u_i\in V_i$ for all $i$, then $D\le d_{H}(u_1,\ldots,u_k)<2D$.  In this case we have
	\begin{equation*}\ex(H,\c{F})=\Om\l(\f{(D)^{1-\be}}{\Del(\log \Del)^5}\r)\cdot e(H).\end{equation*}
\end{prop}
Before proving this, let us quickly show that this result implies Proposition~\ref{prop:coMatch}.

\begin{proof}[Proof of Proposition~\ref{prop:coMatch} assuming Proposition~\ref{prop:coMatchTech}]
	 Let $H$ be as in Proposition~\ref{prop:coMatch} with $V_1\cup \cdots \cup V_r$ an $r$-partition of $H$. Given $I=\{i_1,\ldots,i_k\}\sub [r]$, let $E_I\sub E(H)$ be the set of edges $\{u_1,\ldots,u_r\}$ such that $d_H(u_{i_1},\ldots,u_{i_k})\ge D$.  By the hypothesis of Proposition~\ref{prop:coMatch}, we have \[\half e(H)\le \l|\bigcup E_{I}\r|\le \sum |E_{I}|.\] 
	Thus $|E_I|\ge 2^{-r-1}e(H)$ for some $k$-set $I$, and we will assume $|E_{[k]}|\ge 2^{-r-1}e(H)$ without loss of generality.
	
	Let $E_j\sub E_{[k]}$ denote the set of edges $\{u_1,\ldots,u_r\}$ with $2^jD\le d_H(u_1,\ldots,u_k)<2^{j+1}D$.  The $k$-degree of any $k$-set of $H$ is at most $\Del$, so the $E_j$ sets with $0\le j\le \log \Del$ partition $E_{[k]}$.  By the pigeonhole principle there exists some $j'$ such that
	\[|E_{j'}|\ge \rec{1+\log \Del}|E_{[k]}|\ge \rec{2^r \log \Del}e(H).\]
	
	Let $H'\sub H$ consist of the edges $E_{j'}$ and let $D'=2^j D$.  By construction, if $\{u_1,\ldots,u_r\}\in E(H')$, then $D'\le d_{H'}(u_1,\ldots,u_k)<2D'$.  Lastly, let $\Del'=2 \Del \log \Del$.  Observe that $H'$ has maximum degree at most $\Del\le \Del'$ and that $D'\le \Del\le \Del'/\log \Del'$ for $\Del$ sufficiently large.  Thus $H'$ satisfies the conditions of Proposition~\ref{prop:coMatchTech}, so we have
	\begin{align*}\ex(H,\c{F})&\ge \ex(H',\c{F})=\Om\l(\f{(D')^{1-\be}}{\Del'(\log \Del')^5}\r)\cdot e(H')\\ &\ge \Om\l(\f{D^{1-\be}}{\Del (\log \Del)^6}\r)\cdot \rec{2^r \log \Del} e(H)=\Om\l(\f{D^{1-\be}}{\Del(\log \Del)^{7}}\r)\cdot e(H),\end{align*}
	where we used $(D')^{1-\be}\ge D^{1-\be}$ since $\be \le 1$ and $D'\ge D$. 
\end{proof}

We now prove Proposition~\ref{prop:coMatchTech}.
\begin{proof}[Proof of Proposition~\ref{prop:coMatchTech}]
	Let $H$ be as in the proposition statement.  We assume throughout that $\Del$ is at least as large as some sufficiently large constant, the result being trivial otherwise. Whenever we consider sets $\{u_1,\ldots,u_r\}$ we will always assume $u_i\in V_i$ for all $i$.  We define
	\[p=\f{D\log \Del}{\Del}\le 1.\]
	
	Let $G=\pa H\l[\bigcup_{i\le k} V_i\r]$.  For ease of presentation we refer to edges of $H$ as \textit{hyperedges} and edges of $G$ simply as edges, and in this proof $e$ will always refer to edges of $G$.  Define $G_p\sub G$ to be the $k$-graph obtained by keeping each edge of $G$ independently with probability $p$, and let $H_p\sub H$ comprise all the hyperedges which contain edges of $G_p$.
	\begin{claim}\label{cl:degree}
		For all $u\in V(H)$, \[\Pr[d_{H_p}(u)> 8D(\log \Del)^3]\le \f{2 \log \Del}{\Del^{2}}.\]
	\end{claim}
		\textit{Proof.} Fix some $u\in V(H)$.  Given $e\in E(G)$, let $E_{e}$ be the set of hyperedges $e'\in E(H)$ which contain $e\cup \{u\}$, and let $\1_{e}$ be the indicator variable which is 1 if $e\in E(G_p)$ and 0 otherwise.  Thus $d_{H_p}(u)=\sum_{e}\l|E_{e}\r|\cdot \1_{e}$.  Moreover, if we let $E_j$ be the set of $e\in E(G)$ with $2^j\le |E_{e}|<2^{j+1}$ and let $S_j:=\sum_{e\in E_j}\1_{e}$, we find \[d_{H_p}(u)=\sum_{e} \l|E_{e}\r|\cdot\1_{e}<\sum_j 2^{j+1}\sum_{e\in E_j}\1_{e}=\sum_j 2^{j+1} S_j.\]
		Observe that we only have to consider $0\le j\le \log_2D\le 2\log (\Del)-1$ since each $e$ is contained in less than $2D$ hyperedges and $D\le \Del$ which is sufficiently large.  Thus to prove the result \textbf{}it will be enough to show $\Pr[S_j>2^{1-j} D(\log \Del)^2]\le 2 \Del^{-2}$ for all $j$ in this range and then to apply the union bound to each of these $2\log\Del $ events.
		
		Fix some $j$ and let $\al$ be such that $|E_j|=\al 2^{-j}\Del$, which means \[\E[S_j]=p\cdot \al 2^{-j}\Del=\al 2^{-j}D\log \Del.\]   Note that $\al\le 1$ since $\Del\ge d_H(u)\ge 2^j |E_j|$.  Using this and the Chernoff bound~\eqref{eq:Chern} gives \begin{align*}\Pr[S_j>2^{1-j}D(\log \Del)^2]&\le \Pr[S_j-\al 2^{-j} D\log \Del>2^{-j} D(\log \Del)^2]\\ &=\Pr[S_j-\E[S_j]> \al^{-1}(\log \Del)\cdot \E[S_j]]\\ &\le e^{-\al^{-1} 2^{-j-1}D(\log \Del)^3}\le \Del^{-2},\end{align*}
		where this last step used $\al\le 1,\ 2^{j}\le D$, and that $\Del$ is sufficiently large. \hfill $\blacksquare$\medskip
	
	We use Claim~\ref{cl:degree} to prove the following.
	\begin{claim}\label{cl:H'}
		There exists a subgraph $H'\sub H$ such that $H'$ has maximum degree at most $8D(\log \Del)^3$, every $k$-set $S\sub \bigcup_{i\le k} V_i$ has $d_{H'}(S)=0$ or $d_{H'}(S)\ge D$, and 
		\begin{equation}\label{eq:H'}e(H')\ge \f{D\log \Del}{2\Del}\cdot e(H).\end{equation}
	\end{claim}		
	\textit{Proof.} Let $U$ denote the set of vertices $u$ with $d_{H_p}(u)> 8D(\log \Del)^{3}$ and let $\1_u$ be the indicator variable which is 1 if $u\in U$ and 0 otherwise.  Let $Y$ denote the set of hyperedges of $H_p$ which contain at least one vertex in $U$, and we think of $Y$ as a set of ``bad'' hyperedges.  Observe that \[|Y|\le \sum_{u\in U} d_{H_p}(u)=\sum_u \1_u d_{H_p}(u).\]  Using this, Claim~\ref{cl:degree} and $d_{H_p}(u)\le d_H(u)$ for all $u$, we conclude that \begin{equation}\label{eq:Y}\E[|Y|]\le \f{2 \log \Del}{\Del^{2}}\sum_u d_H(u)=\f{2r \log \Del}{\Del^{2}}\cdot e(H).\end{equation} 
	
	Define \[Y'=\{e\cap \bigcup_{i\le k} V_i:e\in Y\},\hspace{2em} Z=\{e\in E(H):e\cap \bigcup_{i\le k} \in Y'\}.\]  That is, $Y'$ consists of edges of $G$ that are contained in the ``bad'' hyperedges $Y$, and $Z$ consists of the hyperedges of $H$ which contain edges of $Y'$. Observe that $Y\sub Z$ and \begin{equation}\label{eq:Z}|Z|\le 2D|Y'|\le 2D|Y|,
	\end{equation} 
	since each $k$-set of $H_p\sub H$ is  contained at most $2D$ hyperedges.  Let $H'_p\sub H_p$ be the subgraph obtained by deleting every edge in $Z$.  In particular note that because $Y\sub Z$, $H'_p$ has maximum degree at most $8 D(\log \Del)^3$.  Also by definition of $Z$, every $k$-set $S\sub \bigcup_{i\le k} V_i$ has  $d_{H'_p}(S)=0$ or $d_{H'_p}(S)\ge D$, depending on whether $S\in Y'$ or not.  It thus suffices to show that in expectation $H'_p$ has at least as many edges as in \eqref{eq:H'}.  And indeed, each hyperedge of $H$ is in $H_p$ with probability $p$, so we have by \eqref{eq:Y} and \eqref{eq:Z} \begin{align*}\E[e(H'_p)]=\E[e(H_p)-|Z|]&\ge \f{D \log \Del}{\Del}\cdot e(H)-\f{4r D\log \Del}{\Del^2}\cdot e(H)\nonumber\\ &\ge \f{D \log \Del}{2\Del}\cdot e(H),\end{align*}
	where this last step used $\Del$ being sufficiently large. \hfill $\blacksquare$\medskip
	
	Let $H'$ be as in Claim~\ref{cl:H'} and let $G'=\pa H'\l[\bigcup_{i\le k} V_i\r]$. The hypergraph $H'$ has maximum degree at most $8D (\log \Del)^3$ and every $k$-set in $\bigcup_{i\le k} V_i$ has $k$-degree 0 or at least $D$, so $G'$ has maximum degree at most $8(\log \Del)^3$. Similarly $e(G')\ge\half D^{-1} e(H')$ since each edge of $G'\sub G$ is contained in at most $2D$ hyperedges of $H'\sub H$.  By Lemma~\ref{lem:match} there exists a matching $M\sub G'$ with \begin{equation}\label{eq:M}|M|\ge \f{e(H')}{16 k D (\log \Del)^{3}} =\Om\l(\f{1}{\Del (\log \Del)^2}\r)\cdot e(H),\end{equation}
	 where this last step used \eqref{eq:H'}. 
	
	Let $G_M$ be the $(r-k+1)$-graph on $M\cup \bigcup_{i>k}V_i$
	which contains the edge $\{e,u_{k+1},\ldots,u_r\}$ if and only if $e\cup \{u_{k+1},\ldots,u_r\}\in E(H')$.  By \eqref{eq:M} and the fact that each $k$-set of $\bigcup_{i\le k} V_i$ has $k$-degree 0 or at least $D$ in $H'$, we have \[e(G_M)\ge D|M|=\Om\l(\f{D}{\Del(\log \Del)^{2}}\r) \cdot e(H).\] Each $M$ vertex of $G_M$ has degree at most $2D$ since each of these $k$-sets are contained in at most $2D$ hyperedges of $H$, and each $u\in \bigcup_{i>k}V_i$ has degree at most $8D(\log \Del)^3$ by construction of $H'$.  By the hypothesis of the proposition, there exists a $\c{P}_k({\c{F}})$-free subgraph $G''\sub G_M$ with \begin{align}e(G'')&=\ex(G_M,\c{P}_k(\c{F}))= \Om(D^{-\be} (\log \Del)^{-3\be}) \cdot \Om\l(\f{D}{\Del(\log \Del)^{2}}\r) \cdot e(H)\nonumber\\&=\Om\l(\f{D^{1-\be}}{\Del(\log \Del)^{5}}\r)\cdot e(H),\label{eq:coG'}\end{align}
	where this last step used $\be\le 1$.
	
	Finally, let $H''\sub H'$ be the subgraph which contains the hyperedge $e\cup \{u_{k+1},\ldots,u_r\}$ if and only if $\{e,u_{k+1},\ldots,u_r\}$ is an edge of $G''$.   Note that $e(H'')=e(G'')$, so by \eqref{eq:coG'} we will be done if we can show that $H'$ is $\c{F}$-free.  Indeed, $\pa H''\l[\bigcup_{i\le k} V_i\r]\sub M$ is a matching and $\pa H''\l[\bigcup_{i\ge k} V_i\r]\cong G''$ is $\c{P}_k(\c{F})$-free, so the result follows from Lemma~\ref{lem:LFree}.
\end{proof}

\subsection{Examples of $\c{P}_k(\c{F})$}
To apply Proposition~\ref{prop:coMatch}, we need to understand $\c{P}_k(\c{F})$ for families of cycles $\c{F}$. The simplest case is the loose cycle.
\begin{lem}\label{lem:linkL}
	For all $\ell\ge 3$, we have $\c{P}_2(C_\ell^{3})=\emptyset$.
\end{lem}
\begin{proof}
	Let $V_1\cup V_2\cup V_3$ be a tripartition of $C_\ell^{3}$ such that $\pa C_\ell^{3}[V_1\cup V_2]$ is a matching.  If $\{u_1,u_2,u_3\}\in E(C_\ell^{3})$ with $u_i\in V_i$ for all $i$, then this must be the only edge containing $u_1$ since any other edge must also contain $u_2$, contradicting $C_\ell^{3}$ being linear.  Similarly it is the only edge containing $u_2$, so  $\{u_1,u_2,u_3\}$ contains two vertices of degree 1 which does not hold for any edge of $C_\ell^{3}$, a contradiction.
\end{proof}

We recall that $S$ is a sunflower if there exists a set $K$ called the kernel of $S$ such that $e\cap e'=K$ for all distinct $e,e'\in E(S)$.  We let $\c{S}_\ell^{r}$ denote the set of $r$-uniform sunflowers on $\ell$ edges.  We let $\BB{\ell}{r}:=\B{\ell}{r}\sm \c{S}_\ell^r$ and $\BB{[\ell]}{r}=\bigcup_{\ell'=2}^\ell \BB{\ell'}{r}$, which agrees with our definition in Theorem~\ref{thm:BGen}.  

We define $\S{\ell}{r}$ to be the set of $r$-graphs $\widetilde{S}$ that can be obtained by adding an edge $e$ to some $S\in \c{S}_{\ell}^r$ such that $K\cap e\ne \emptyset$ where $K$ is the kernel of $S$, and such that $S\cup e$ is not a sunflower.  Note that we do not require $V(\widetilde{S})=V(S)$, that is, $e$ might include vertices which are not in the original sunflower $S$.  For any $\widetilde{S}\in \S{\ell}{r}$, there exists an edge $e$ such that $\widetilde{S}-e$ is a sunflower, and we call $e$ an \textit{extra edge} of $\widetilde{S}$ (which may not be unique), and we define the \textit{kernel} of $\widetilde{S}$ (with respect to the extra edge $e$) to be the kernel of the sunflower $\widetilde{S}-e$. Finally, we define $\S{[\ell]}{r}=\bigcup_{\ell'=2}^{\ell} \S{\ell'}{r}$.

\begin{lem}\label{lem:linkB}
	Every element of $\c{P}_k(\BB{[\ell]}{r}\cup \S{[\ell]}{r})$ contains an element of $\BB{[\ell]}{r-k+1}\cup \S{[\ell]}{r-k+1}$ as a subgraph.
\end{lem}
\begin{proof}
	We prove the result for any fixed $r\ge 2$ by induction on $k$.  We start with $k=2$, and for ease of presentation we define $F_\pa:=\pa F\l[\bigcup_{i\ge 2}V_i\r]$ whenever $F$ has an $r$-partition $V_1\cup \cdots \cup V_r$ that is clear from context.  
	\begin{claim}\label{cl:sunflower}
		If $F\in \S{\ell'}{r}$, then $\c{P}_2(F)\sub \S{\ell'}{r-1}$.
	\end{claim}
	\textit{Proof.}  Let $e$ be an extra edge of $F$ and $K$ the kernel of $F':=F-e$. Let $V_1\cup \cdots \cup V_r$ be an $r$-partition of $F$ such that $V_1\cup V_2$ induce a matching.  Observe that $F'_\pa$ is a sunflower with kernel $K\sm V_1$, so $F_\pa$ is this sunflower together with the edge $e_\pa:=e\sm V_1$.  We claim that $e_\pa$ intersects $K\sm V_1$.  Indeed, by definition of $ \S{\ell'}{r}$ we have $u\in K\cap e$ for some vertex $u$, and if $u\notin V_1$ then the subclaim holds.  If $u\in V_1$, then $V_1\cup V_2$ inducing a matching implies that there is some $u'\in V_2$ such that every edge containing $u$ also contains $u'$.  In particular, $u'\in e_\pa \cap (K\sm V_1)$, proving the subclaim.
	
	It remains to show that $F_\pa$ is not a sunflower. By definition of $F$, there exists an edge $e'\ne e$ such that $e\cap e'\ne K$, and we will be done if we can show $(e\cap e')\sm V_1\ne K\sm V_1$.    Assume for the sake of contradiction that $(e\cap e')\sm V_1= K\sm V_1$, and this together with $e\cap e'\ne K$ implies that there exists $u_1\in V_1$ such that $u_1$ is in exactly one of $e\cap e'$ or $K$.  If $u_1\in e\cap e'$, then $V_1\cup V_2$ inducing a matching implies that there is some $u_2\in V_2$ such that every edge either contains both of $u_1,u_2$ or neither of them.   In particular, we have $u_2\in e\cap e'$ and $u_2\notin K$ since $u_1\notin K$, contradicting $(e\cap e')\sm V_1=K\sm V_1$.  Thus we must have $u_1\in K$, which similarly implies the existence of some $u_2\in K\cap V_2$ with $u_2\notin e\cap e'$, a contradiction.  We conclude that $F_\pa$ is not a sunflower, and hence $F_\pa\in \S{\ell'}{r-1}$ as desired. \hfill $\blacksquare$\medskip

	Recall that $A=\{v_1,\ldots,v_{\ell'}\}\sub V(F)$ is a core set if there exist edges $e_1,\ldots,e_{\ell'}\in E(F)$ with $v_i,v_{i+1}\in e_i$ for all $1\le i<\ell$ and $v_1,v_{\ell'}\in e_{\ell'}$.
	\begin{claim}\label{claim1}
		Let $\hat{F}\in \BB{\ell'}{r}$ with $\hat{A}=\{v_1,\ldots,v_{\ell'}\}$ a core set of $\hat{F}$, and let $V_1\cup \cdots \cup V_r$ be an $r$-partition of $\hat{F}$ such that $V_1\cup V_2$ induce a matching $M$.  If $\{v_i,v_j\}\notin M$ for all $i,j$, then $\hat{F}_\pa \in \BB{\ell'}{r-1}$.
	\end{claim}
	\textit{Proof.} For $\{x,y\}\in M$, define $\bar{x}=y$ and $\bar{y}=x$.  Let $w_i=v_i$ if $v_i\notin V_1$ and $w_i=\bar{v}_i$ otherwise. By the hypothesis of the claim, $\bar{A}:=\{w_1,\ldots,w_{\ell'}\}$ consists of $\ell'$ distinct vertices of $V(\hat{F}_\pa)=V_2\cup \cdots \cup V_r$.  Since every edge containing $v_i$ also contains $\ol{v_i}$, this defines a core set for $\hat{F}_\pa$, and hence $\hat{F}_\pa\in \B{\ell'}{r-1}$.
	
	It remains to show that $\hat{F}_\pa$ is not a sunflower.  Assume for contradiction that it is a sunflower with kernel $K$.  If there exists $u_2\in K\cap V_2$, then because $V_1\cup V_2$ induce a matching and every edge of $F_\pa$ contains $u_2$, there exists some $u_1\in V_1$ contained in every edge of $F$, so $F$ is a sunflower with kernel $K\cup \{u_1\}$, contradicting that $\hat{F}\in \BB{\ell'}{r}$.  Thus every vertex in  $V_2$ is in exactly one edge of the sunflower $\hat{F}_\pa$, and again $V_1\cup V_2$ inducing a matching implies that each vertex of $V_1$ is contained in exactly one edge.  We conclude that $\hat{F}$ is a sunflower with kernel $K$, a contradiction. \hfill $\blacksquare$\medskip
	
	We now return to the proof of Lemma~\ref{lem:linkB}.  Assume for contradiction that there exists some $2\le \ell'\le \ell$ and $F\in  \BB{\ell'}{r}$ which has an $r$-partition $V_1\cup \cdots \cup V_r$ such that $V_1\cup V_2$ induce a matching $M$ and such that $F_\pa$ does not contain an element of $\BB{[\ell']}{r-1}\cup \S{ [\ell']}{r-1}$ as a subgraph.  We further choose $\ell'$ to be the smallest integer such that there exists such an $F\in  \BB{\ell'}{r}$.  Observe that $\BB{2}{r}$ is empty, so $\ell'\ge 3$.  Let $\{v_1,\ldots,v_{\ell'}\}$ be a core set of $F$ with respect to its edges $e_1,\ldots,e_{\ell'}$.  
	
	If $\{v_i,v_j\}\notin M$ for all $i\ne j$, then Claim~\ref{claim1} shows that $F_\pa\in \BB{\ell'}{r-1}$, a contradiction.  Thus by reindexing we will assume  $\{v_{1},v_{i}\}\in M$ for some $i>1$.

	\begin{claim}\label{cl:twoSuns}
			Let $F'\sub F$ consist of the edges $e_{1},\ldots,e_{i-1}$ and let $F''=F\sm F'$ consist of the edges $e_{i},\ldots,e_{\ell'}$.  Then $F',F''$ are sunflowers (possibly with one edge) which have kernels $K',K''$ containing $v_1$.
	\end{claim}
	\textit{Proof.} We only prove the result for $F'$, the proof for $F''$ being completely analogous.  If $i=2$ then $F'$ has a single edge which contains $v_1$, so we will assume $i>2$.  We claim that $A'=\{v_1,\ldots,v_{i-1}\}$ is a core set of $F'$ with respect to $\{e_1,\ldots,e_{i-1}\}$.  Indeed, because $\{v_1,\ldots,v_{\ell'}\}$ is a core set with respect to $\{e_1,\ldots,e_{\ell'}\}$, we automatically have $v_j,v_{j+1}\in e_j$ for $1\le j\le i$, so $A'$ will be a core set provided $v_{1}\in e_{i-1}$.  This follows since $v_{i-1}\in e_{i-1}$ and $\{v_{1},v_{i-1}\}\in M$ implies that every edge containing $v_{i-1}$ also contains $v_{1}$.
	
	Thus $F\in \B{i-1}{r}$.  If we assume for contradiction that $F'$ is not a sunflower, then $F'\in \BB{i-1}{r}$.  By the minimality of $\ell'>i-1$, we have that $F'_\pa$ contains an element of $\BB{[\ell']}{r-1}\cup \S{[\ell']}{r-1}$ as a subgraph.  This will also be a subgraph of $F_\pa \supseteq F'_\pa$, a contradiction.  Thus it must be that $F'$ is a sunflower.  Moreover, because $v_1$ is in a core set of $F'$, it must have degree at least two in $F'$, and hence it must be in its kernel $K'$. \hfill $\blacksquare$\medskip
	
	By Claim~\ref{cl:twoSuns}, we see that $F$ is the union of two sunflowers $F',F''$ intersecting at a common vertex $v_1\in K'\cap K''$, so in particular every edge in $F$ contains $v_1$.  Since $\ell'\ge 3$ we will assume without loss of generality that $F'$ has at least two edges, i.e. that $i\ge 3$.  Let $F^{(j)}\sub F$ be the subgraph with edge set $\{e_1,\ldots,e_j\}$.  Since $F^{(\ell')}=F$ is not a sunflower and $F^{(i-1)}=F'$ is, there exists a largest $j$ with $i-1\le j<\ell'$ such that $F^{(j)}$ is a sunflower on at least two edges.  Observe that $F^{(j)}$ has kernel $K$ since $F'\sub F^{(j)}$ has this kernel, and that $F^{(j+1)}\in \S{j+1}{r}\sub \S{[\ell]}{r}$ since $v_1\in e_{j+1}\cap K$.  By Claim~\ref{cl:sunflower} we have that $F_\pa^{(j+1)}\sub F_\pa$ contains an element of $\S{[\ell]}{r-1}$ as a subgraph, a contradiction. 
	
	In total we conclude that if $F\in \BB{[\ell]}{r}\cup \S{[\ell]}{r}$, then every element of $\c{P}_2(F)$ has a subgraph in $\BB{[\ell]}{r-1}\cup \S{[\ell]}{r-1}$.  This proves the lemma when $k=2$.  Assume that we have proven the lemma up to some value $k>2$.  Observe that for all $F$ we have $\c{P}_k(F)\sub \c{P}_{k-1}(\c{P}_2(F))$. By the $k=2$ case of the lemma, if $F\in \BB{[\ell]}{r}\cup \S{[\ell]}{r}$, then every $F_2\in \c{P}_2(F)$ has a subgraph $F'_2\in \BB{[\ell]}{r-1}\cup \S{[\ell]}{r-1}$, and by the inductive hypothesis every element of $\c{P}_{k-1}(F_2')$ contains a subgraph $F_2''$ in $\BB{[\ell]}{r-k+1}\cup \S{[\ell]}{r-k+1}$.  Thus every element of $\c{P}_{k-1}(\c{P}_2(F))\supseteq \c{P}_k(F)$ contains a subgraph in $\BB{[\ell]}{r-k+1}\cup \S{[\ell]}{r-k+1}$, which completes the inductive step of the proof, giving the desired result.
\end{proof}
We make one more observation about $\S{[\ell]}{r}$.
\begin{lem}\label{lem:nonlinSun}
	Every element of $\S{[\ell]}{r}$ is non-linear.
\end{lem}
\begin{proof}
	Let $\widetilde{S}\in \S{[\ell]}{r}$ have extra edge $e$ and kernel $K$.  If $|K|\ge 2$, then in particular at least two edges of $\widetilde{S}$ share a common pair of vertices, so we will assume $|K|\le 1$.  We actually must have $|K|=1$ since $e\cap K\ne \emptyset$.  If every edge $e'\in \widetilde{S}\sm e$ has $e'\cap e=K$, then $\widetilde{S}$ is a sunflower, which is a contradiction to how $\S{[\ell]}{r}$ is defined.  Thus there exists some edge $e'$ with $e'\cap e\ne K$, and since every edge of $\widetilde{S}$ contains the vertex of $K$, we must have that $|e'\cap e|\ge 2$.
\end{proof}

\section{Proof of Theorem~\ref{thm:BGen} and a Lower Bound for $\ex(H,\B{5}{3})$}\label{sec:B}
We prove Theorem~\ref{thm:BGen} by induction on $r$.  Note that $\BB{[\ell]}{2}=\{C_3,C_4,\ldots,C_\ell\}$, and that $\S{[\ell]}{2}=\emptyset$ (this latter case can either be seen directly from the definition or from Lemma~\ref{lem:nonlinSun}).  With this we restate a result of Perarnau and Reed~\cite{PR} in a form that will be convenient to us.
\begin{prop}\label{prop:r2}
	If \eqref{eq:girth} holds for $\ell$ and $r=2$, then for all graphs $G$ with maximum degree at most $\Del$,
	\[\ex(G,\BB{[\ell]}{2}\cup \S{[\ell]}{2})\ge \Del^{-1+\rec{\floor{\ell/2}}-o(1)}\cdot e(G).\]
\end{prop}
This result can also be derived from Proposition~\ref{prop:homGen}, which is essentially a generalization of a theorem of Perarnau and Reed~\cite{PR}.  

It is known that the hypothesis of Proposition~\ref{prop:r2} holds for $\ell=4,5$, see for example \cite{FS}. The $\ell=5$ case allows us to prove the following.

\begin{thm}
	If $H$ is a 3-graph with maximum degree at most $\Del$, then \begin{align*}
		\ex(H,\B{5}{3})&\ge \Del^{-3/4-o(1)}\cdot e(H).\end{align*}
\end{thm}
\begin{proof}
	Let $H$ be a $3$-graph with maximum degree at most $\Del$, and let $H_2\sub H$ be a $3$-partite 3-graph with at least $3^{-3}e(H)$ edges.  First assume that at least half the edges of $H_2$ contain a 2-set with 2-degree at least $\Del^{1/2}$.  Using $\B{5}{3}\sub \BB{[5]}{3}$ (i.e. $\B{5}{3}$ contains no sunflowers), Lemma~\ref{lem:linkB}, Proposition~\ref{prop:r2}, and Proposition~\ref{prop:coMatch} applied to $H_2$ with $\be=\half+\ep$ for any $\ep>0$ gives
	\[\ex(H,\B{5}{3})\ge \ex(H_2,\BB{[5]}{3}\cup \S{[5]}{3})\ge (\Del^{1/2})^{1/2-\ep}\cdot \Del^{-1-o(1)}\cdot e(H_2)=\Del^{-3/4-\ep/2-o(1)}\cdot e(H),\]
	and taking $\ep$ arbitrarily small gives the desired bound.
	
	From now on we assume at most half the edges of $H_2$ contain a 2-set with 2-degree larger than $\Del^{1/2}$.  Let $H_3\sub H_2$ be the subgraph consisting of the edges which do not contain a 2-set with 2-degree larger than $\Del^{1/2}$, so \[e(H_3)\ge \half e(H_2)\ge \rec{54}e(H).\]  It was shown by Ergemlidze, Gy\H{o}ri, and Methuku~\cite{EGM} that $\ex(t,\B{2}{3}\cup \B{5}{3})=\Theta(t^{3/2})$.  Using this and Lemma~\ref{lem:homB}(2) gives \[\ex(t,\HI{\B{5}{3}})\ge \ex(t,\B{2}{3}\cup \B{5}{3})=\Om(t^{3/2}).\]  The hypergraph $H_3$ has maximum 2-degree at most $\Del^{1/2}$ by consturction, so we can apply Proposition~\ref{prop:homGen} to $H_3$ with $t=\Theta(\Del^{1/2})$, giving
	\[\ex(H,\B{5}{3})\ge \ex(H_3,\B{5}{3})= \Om(\Del^{-3/4}) \cdot e(H_3)=\Om(\Del^{-3/4}) \cdot e(H),\]
	as desired.
\end{proof}

We now prove Theorem~\ref{thm:BGen} using a similar argument.

\begin{proof}[Proof of Theorem~\ref{thm:BGen}]
Fix $\ell\ge 3$.  We prove by induction on $r$ the stronger fact that if \eqref{eq:girth} holds for $\ell,r$, then for all $r$-graphs $H$ with maximum degree at most $\Del$, we have 
\begin{align}\ex(H,\BB{[\ell]}{r}\cup \S{[\ell]}{r})\ge \Del^{-1+\f{1}{(r-1)\floor{\ell/2}}-o(1)}\cdot e(H).\label{eq:inductB}\end{align}
The case $r=2$ follows from Proposition~\ref{prop:r2}.  Assume \eqref{eq:inductB} holds up to $r$.  

We claim that  \eqref{eq:girth} holding for $\ell,r$ also implies that \eqref{eq:girth} holds for $\ell$ and any $2\le k\le r-1$.  Indeed, given an extremal $\B{[\ell]}{r}$-free $r$-graph $J$, we form a $\B{[\ell]}{k}$-free $k$-graph $J'$ by including some $k$-set from each edge of $J$, and each of these $k$-sets will be distinct because $J$ is linear. With this $e(J')=e(J)$, and it is straightforward to see that $J'$ will be $\B{[\ell]}{k}$-free if $J$ is $\B{[\ell]}{r}$-free, proving the claim. Thus by the inductive hypothesis, if \eqref{eq:girth} holds for $\ell,r$ then for all $2\le k\le r-1$ we have \begin{equation}\ex(H',\BB{[\ell]}{r-k+1}\cup \S{[\ell]}{r-k+1})\ge \Del^{-1+\f{1}{(r-k)\floor{\ell/2}}-o(1)}\cdot e(H'),\label{eq:inductB2}\end{equation} 
whenever $H'$ is an $(r-k+1)$-graph with maximum degree at most $\Del$.

Assume that \eqref{eq:girth} holds for $\ell,r$.  Let $H$ be an $r$-graph with maximum degree at most $\Del$.  We will prove by induction on $2\le k\le r$ that there exists a subgraph $H_k\sub H$ such that $H_k$ is $r$-partite, $e(H_k)\ge r^{-r} 2^{-k}e(H)$, and $H_k$ has maximum $p$-degree at most \[D_p:=\Del^{(r-p)/(r-1)}\] for all $p<k$.  This holds for $k=2$ by taking $H_2\sub H$ to be an $r$-partite subgraph with at least $r^{-r}e(H)$ edges.  Assume the result has been proven through $k$.  If $H_{k}$ contains at least $\half e(H_{k})\ge r^{-r}2^{-k-1}e(H)$ edges which do not contain a $k$-set with $k$-degree larger than $D_k$, then taking $H_{k+1}$ to consist of these edges gives the desired subgraph.  Thus we will assume that at least half the edges of $H_k$ contain a $k$-set with $k$-degree larger than $D_k$.  By \eqref{eq:inductB2} and Lemma~\ref{lem:linkB}, we can apply  Proposition~\ref{prop:coMatch} to $H_k$ with \[\be=1-\f{1}{(r-k)\floor{\ell/2}}+\ep\] for any $\ep>0$ to conclude
\[\ex(H_k,\BB{[\ell]}{r}\cup \S{[\ell]}{r})\ge D_k^{\f{1}{(r-k)\floor{\ell/2}}-\ep}\cdot \Del^{-1-o(1)}\cdot e(H_2)=\Del^{-1+\f{1}{(r-1)\floor{\ell/2}}-\f{r-k}{r-1}\ep-o(1)}\cdot e(H),\]
and taking $\ep$ arbitrarily small gives \eqref{eq:inductB} as desired.  Thus we may assume such a subgraph $H_{k+1}$ exists.

In total we have found a subgraph $H_r\sub H$ with $e(H_r)=\Om(e(H))$ which has maximum $k$-degree at most $D_k$ for all $k$.  By Lemmas~\ref{lem:nonlinSun} and \ref{lem:homB}(1), every element of $\HI{\S{[\ell]}{r}}$ contains an element of $\B{2}{r}$ as a subgraph.  Thus by Lemma~\ref{lem:homB}(3), every element of $\HI{\BB{[\ell]}{r}\cup \S{[\ell]}{r}}$ contains an element of $\B{[\ell]}{r}$ a subgraph, so \[\ex(t,\HI{\BB{[\ell]}{r}\cup \S{[\ell]}{r}})\ge\ex(t,\B{[\ell]}{r})\ge t^{1+\rec{\floor{\ell/2}}-o(1)},\] 
with this last step using the assumption that \eqref{eq:girth} holds for $\ell,r$.  We then apply Proposition~\ref{prop:homGen} with $t=\Theta(\Del^{1/(r-1)})$ and conclude \[\ex(H_r,\BB{[\ell]}{r}\cup \S{[\ell]}{r})\ge \Del^{-1+\f{1}{(r-1)\floor{\ell/2}}-o(1)}\cdot e(H)\] as desired.
\end{proof}

\section{Proof of Theorem~\ref{thm:disjoint}}\label{sec:disjoint}
Recall that $\F\in \B{4}{3}$ is the hypergraph from Figure~\ref{fig}. To improve upon the trivial bound $\ex(H,\F)\ge \ex(H,\B{4}{3})\ge \Del^{-3/4-o(1)}e(H)$ from Corollary~\ref{cor:B}, we use the following counting lemma.
\begin{lem}\label{lem:countF}
	Let $\c{N}_{\F}(H)$ denote the number of copies of $\F$ in the 3-graph $H$.  If $H$ has maximum degree at most $\Del$ and maximum 2-degree at most $D$, then
	\[\c{N}_{\F}(H)\le 9D\Del\cdot e(H).\]
\end{lem}
\begin{proof}
	Consider all of the ways of choosing a 4-tuple $(e_1,e_2,e_3,e_4)$ of edges of $H$ such that there exists an isomorphism $\phi$ from these edges to $\F$ such that $\phi(e_1)=\{u_1,u_2,u_3\}, \phi(e_2)=\{u_1,u_2,v_1\},\ \phi(e_3)=\{v_1,v_2,v_3\},\ \phi(e_4)=\{u_3,v_2,v_3\}$.  Observe that $\c{N}_{\F}(H)$ is at most the number of such 4-tuples, so it suffices to bound this quantity.
	
	Any of the $e(H)$ edges of $H$ can be $e_1$ in such a 4-tuple.  Given $e_1$, $e_2$ must be one of the at most $3D$ edges intersecting $e_1$ in a pair, and given this $e_3$ must be one of the at most $3\Del$ edges intersecting $e_2$ in a vertex.  Once these three edges have been chosen, $e_4$ is uniquely determined, so we conclude the desired bound.
\end{proof}

\begin{proof}[Proof of Theorem~\ref{thm:disjoint}]	
	Let $H'\sub H$ be a 3-partite 3-graph with at least $3^{-3} e(H)$ edges.  If more than half of the edges of $H'$ contain a pair of vertices with 2-degree at least $\Del^{4/5}$, then by using Lemmas~\ref{lem:linkB} and Proposition~\ref{prop:r2}, we apply Proposition~\ref{prop:coMatch} with $k=2$ and $\be=1/2+o(1)$ to find \[\ex(H,\F)\ge \ex(H',\F)\ge \Om\l(\f{(\Del^{4/5})^{1-\be}}{\Del (\log \Del)^7}\r) \cdot e(H)= \Del^{-3/5-o(1)}e(H).\] 
	
	If instead at most half of the edges of $H'$ contain a pair with 2-degree at least $\Del^{4/5}$, then we can let $H''\sub  H'$ consist of all of these edges.  Let $H_p\sub H''$ be the random 3-graph obtained by keeping each edge of $H''$ independently and with probability $p=\rec{9}\Del^{-3/5}$.  Observe that
	\begin{align*}
		\E[e(H_p)]=&p\cdot e(H'')\ge \rec{486}\Del^{-3/5} \cdot e(H),\tr{ and}\\ 
		\E[\c{N}_{\F}(H_p)]=&p^4 \cdot \c{N}_{\F}(H'')\le 9^{-4} \Del^{-12/5}\cdot 9 \Del^{9/5} e(H)=\rec{729}\Del^{-3/5} \cdot e(H),
	\end{align*}
	where this last inequality used Lemma~\ref{lem:countF} and that $H''$ has maximum 2-degree at most $D=\Del^{4/5}$ and that it has at most $e(H)$ edges.  In particular, there exists some specific choice of $H'''\sub H''$ with $e(H''')-\c{N}_{\F}(H''')=\Om(\Del^{-3/5}e(H))$.  Deleting an edge from each copy of $\F$ in $H'''$ gives an $\F$-free hypergraph with the desired number of edges.
\end{proof}

\section{Proofs of Theorem~\ref{thm:loose} and Proposition~\ref{prop:linear}}\label{sec:L}

To prove Theorem~\ref{thm:loose} and Proposition~\ref{prop:linear} we require the following lemmas.

\begin{lem}\label{lem:looseCount}
	Let $H$ be an $r$-graph with maximum degree $\Del$ and maximum 2-degree $D$.  If $\c{N}_\ell(H)$ denotes the number of copies of $C_\ell^r$ in $H$, then
	\[\c{N}_{\ell}(H)\le r^{\ell} D\Del^{\ell-2}\cdot e(H).\]
\end{lem}
\begin{lem}\label{lem:linear}
	If $H$ is an $r$-graph with maximum 2-degree at most $D$, then there exists a linear subgraph $H'\sub H$ with $e(H')\ge e(H)/r^2D$.
\end{lem}
The proof of Lemma~\ref{lem:looseCount} is a straightforward counting argument as in the proof of Lemma~\ref{lem:countF}, and the proof of Lemma~\ref{lem:linear} uses a greedy construction.  We omit the details.

Lastly, we require a special case of a theorem of Mubayi and Yepremyan~\cite{MY} for random hypergraphs. We recall that $H_{n,p}^r$ is the random $r$-graph on $n$ vertices obtained by including each edge of $K_n^r$ independently and with probability $p$, and that a statement depending on $n$ holds \textit{asymptotically almost surely}, or simply a.a.s., if it holds with probability tending to 1 as $n$ tends to infinity.
\begin{prop}[\cite{MY}]\label{prop:randLoose}
	For all even $\ell\ge 4$ and $r\ge 3$, at $p=n^{-r+2}$ we have a.a.s.
	\[\ex(H_{n,p}^{r},C_{\ell}^{r})\le n^{1+\rec{\ell-1}+o(1)}.\]
\end{prop}

\begin{proof}[Proof of Theorem~\ref{thm:loose}]
	The proof is very similar to that of Theorem~\ref{thm:disjoint}, so we will omit some of the redundant details.  By losing at most a constant fraction of the edges, we can assume $H$ is $r$-partite.  If more than half of its edges contain a pair with 2-degree at least $D=\Del^{1/\ell}$, then using Lemma~\ref{lem:linkL} we apply Proposition~\ref{prop:coMatch} with $\be=0$ to find $\ex(H,C_\ell^{3})\ge  \Del^{-1+1/\ell-o(1)}e(H)$. 
	
	Thus by losing at most half of its edges we will assume $H$ has maximum 2-degree at most $D$.  Let $H_p$ be $H$ after keeping each edge with probability \[p=3^{-1-2/(\ell-1)} \Del^{-1+1/(\ell-1)} D^{-1/(\ell-1)}.\]  By Lemma~\ref{lem:looseCount}, we find
	\begin{equation}\E[\c{N}_{\ell}(H_p)]\le 3^\ell \Del^{\ell-2}Dp^\ell \cdot e(H)= 3^{\f{-2\ell }{\ell-1}} \Del^{-1+\rec{\ell}}\cdot e(H)  =3^{-1}\cdot \E[e(H_p)].\label{eq:loose}\end{equation}
	Thus if we define $H'_p\sub H_p$ by deleting an edge from  each copy of $C_\ell^{r}$, then we get a subgraph which is $C_\ell^r$-free and which has $\Om(\Del^{-1+1/(\ell-1)})e(H)$ edges in expectation, proving the lower bound.
	
	For the construction, take $p=n^{2-r}$ and $H= H_{n,p}^{r}$.  A straightforward argument using the Chernoff bound \eqref{eq:Chern} gives, for this choice of $p$, that a.a.s.\  $e(H_{n,p}^{r})=\Theta(n^2)$ and that $H_{n,p}^{r}$ has maximum degree $\Del=\Theta(n)$.  By Proposition~\ref{prop:randLoose} we have a.a.s.\[\ex(H_{n,p}^r,C_{\ell}^{r})\le n^{1+\rec{\ell-1}+o(1)}=\Del^{-1+\rec{\ell-1}+o(1)}\cdot e(H_{n,p}^r).\]
	In particular, such an $H$ exists for $n$ sufficiently large, giving the desired result.
\end{proof}

\begin{proof}[Proof of Proposition~\ref{prop:linear}]
	The proof is nearly identical to that of Theorem~\ref{thm:loose}.  For the lower bound, we use that $H$ has maximum 2-degree at most 1, so the exact same computation for the proof of \eqref{eq:loose} works by taking $p=\Theta(\Del^{-1+1/(\ell-1)})$.  For the construction, it is straightforward to use the Chernoff bound \eqref{eq:Chern} to show that a.a.s. $H_{n,p}^r$ at $p=n^{2-r}$ has maximum 2-degree at most $O(\log n)$, so we can apply Lemma~\ref{lem:linear} to find a large linear subgraph $H\sub H_{n,p}^r$ and the rest of the proof works out as before.
\end{proof}

\section{Concluding Remarks}\label{sec:concluding}

$\bullet$ Using arguments analogous to the ones used throughout the paper, it is not too difficult to show for any $F\in \B{4}{3}\sm \{C_4^3,\F\}$ that $\ex(H,F)\ge \Del^{-1/2-o(1)}e(H)$ whenever $H$ has maximum degree at most $\Del$, and this is often tight by considering $H$ to be the clique.  However, our best bounds for $C_4^3$ and $\F$ given by Theorems~\ref{thm:loose} and \ref{thm:disjoint} still have significant gaps.

$\bullet$ Theorem~\ref{thm:disjoint} shows that $\ex(H,\F)\ge \Del^{-3/5-o(1)}e(H)$ for any host $H$ with maximum degree at most $\Del$.  It is known that $\ex(n,\F)=\Theta(n^2)$, so taking $H=K_{\sqrt{\Del}}^3$ gives $\ex(H,\F)=O(\Del^{-1/2}) e(H)$.  It is unclear which of these bounds is closer to the truth, and we leave the following as an open problem.

\begin{prob}
	Determine whether \[\ex(H,\F)\ge \Del^{-1/2-o(1)}\cdot e(H)\] whenever $H$ is a 3-graph with maximum degree $\Del$.
\end{prob}

$\bullet$ While Proposition~\ref{prop:linear} gives essentially tight bounds for loose even cycles in linear hosts, our bounds are far from tight for general hosts.  In particular, we do not know tight bounds for $C_4^3$.
\begin{prob}
	Determine whether there exists a 3-graph $H$ with maximum degree $\Del$ and \[\ex(H,C_4^3)\le \Del^{-3/4+o(1)}\cdot e(H).\]
\end{prob}
A reasonable candidate for such an $H$ is $H_{n,p}$ with $p=n^{-2/3}$.  It is conjectured by Mubayi and Yepremyan~\cite{MY} that $\ex(H_{n,p}^3,C_4^3)\le n^{4/3+o(1)}$ a.a.s.\ at $p=n^{-2/3}$, and this can be rephrased as saying $\ex(H_{n,p}^3,C_4^3)\le \Del^{-3/4+o(1)}e(H)$ a.a.s.  Conversely, if there existed a method which improved upon the known lower bounds for $\ex(H_{n,p}^3,C_4^3)$, then it is possible that this method could also be used to improve lower bounds for general hosts $H$.

$\bullet$ Variants of Propositions~\ref{prop:homGen} and \ref{prop:coMatch} were used to prove essentially tight bounds on the relative Tur\'an numbers for certain kinds of complete $r$-partite $r$-graphs~\cite{SV-K22}, and we suspect that these propositions will continue to be of use for future investigations into relative Tur\'an numbers of hypergraphs.

\textbf{Acknowledgments.}  The authors thank the anonymous referees for their careful reading of this paper and their many helpful comments.

\bibliographystyle{abbrv}

\end{document}